\documentclass[11pt]{amsart}
\usepackage[margin=3cm]{geometry}               
\usepackage[usenames,dvipsnames,svgnames,table]{xcolor}    
\usepackage{graphicx}
\usepackage{tikz}
\usepackage{tikz-cd}
\usetikzlibrary{decorations.pathreplacing}
\usepackage{lineno, hyperref}
\usepackage{amssymb,amsfonts, amsmath, amsthm}
\usepackage{epstopdf}
\usepackage[shortlabels]{enumitem}
\usepackage{ulem}
\usepackage{cancel}
\usepackage{thmtools} 
\usepackage{mathtools} 
\usepackage{mathrsfs}
\usepackage{yhmath}
\usepackage[english]{babel}
\usepackage[toc,page]{appendix}

\swapnumbers
\newtheorem*{theorem*}{Theorem}
\newtheorem*{question*}{Question}
\newtheorem*{corollary*}{Corollary}

\newtheorem{theorem}[subsubsection]{Theorem}
\newtheorem{lemma}[theorem]{Lemma}

\newtheorem{proposition}[theorem]{Proposition}
\newtheorem{corollary}[theorem]{Corollary}

\theoremstyle{definition}

\newtheorem{remark}[theorem]{Remark}
\newtheorem{example}[theorem]{Example}

\newtheorem{question}[theorem]{Question}

\newcommand{\Q}{\mathbb Q}

\newcommand{\cU}{\mathcal U}

\newcommand{\cL}{\mathcal{L}}
\newcommand{\res}{\mathrm{res}}

\newcommand{\Div}{\mathrm{div}}
\newcommand{\Def}{\mathrm{def}}

\newcommand{\cc}{\mathbf{c}}

\newcommand{\sS}{\mathbf{D}}

\newcommand{\eq}{\mathrm{eq}}
\newcommand{\dcl}{\mathrm{dcl}}
\newcommand{\Deff}{\mathrm{def}}
\newcommand{\acl}{\mathrm{acl}}

\newcommand{\ACVF}{\mathrm{ACVF}}
\newcommand{\DOAG}{\mathrm{DOAG}}
\newcommand{\RCVF}{\mathrm{RCVF}}
\newcommand{\PCF}{p\mathrm{CF}}
\newcommand{\CODF}{\mathrm{CODF}}

\newcommand{\Th}{\mathrm{Th}}

\newcommand{\D}[1]{{S^\mathrm{def}_{#1}}}
\newcommand{\tp}{\mathrm{tp}}

\newcommand{\SE}{\mathcal{S}\mathcal{E}}

\makeatletter
\def\widebreve{\mathpalette\wide@breve}
\def\wide@breve#1#2{\sbox\z@{$#1#2$}%
     \mathop{\vbox{\m@th\ialign{##\crcr
\kern0.08em\brevefill#1{0.8\wd\z@}\crcr\noalign{\nointerlineskip}%
                    $\hss#1#2\hss$\crcr}}}\limits}
\def\brevefill#1#2{$\m@th\sbox\tw@{$#1($}%
  \hss\resizebox{#2}{\wd\tw@}{\rotatebox[origin=c]{90}{\upshape(}}\hss$}
\makeatletter

\title{Pro-definability of spaces of definable types}

\author[Pablo Cubides Kovascics]{Pablo Cubides Kovacsics}
\address{Pablo Cubides Kovacsics, Mathematisches Institut der Heinrich-Heine-Universit\"at D\"usseldorf, 
Universit\"atsstr. 1, 40225 D\"usseldorf, Germany. }
\email{cubidesk@hhu.de}

\author[Jinhe Ye]{Jinhe Ye}
\address{Jinhe Ye, University of Notre Dame\\
Department of Mathematics\\
255 Hurley, Notre Dame\\
IN 46556, USA.
}
\email{jye@nd.edu}

\subjclass[2010]{Primary 12L12, Secondary 03C64, 12J25}

\keywords{Pro-definability, definable types, stably embedded pairs}

\begin{document}
\setcounter{secnumdepth}{4}

\maketitle

\begin{abstract} We show pro-definability of spaces of definable types in various classical complete first order theories, including complete o-minimal theories, Presburger arithmetic, $p$-adically closed fields, real closed and algebraically closed valued fields and closed ordered differential fields. Furthermore, we prove pro-definability of other distinguished subspaces, some of which have an interesting geometric interpretation. 

Our general strategy consists in showing that definable types are uniformly definable, a property which implies pro-definability using an argument due to E. Hrushovski and F. Loeser. Uniform definability of definable types is finally achieved by studying classes of stably embedded pairs. 
\end{abstract}

\setcounter{tocdepth}{1}



\newcounter{eqn}

\normalem 
\section{Introduction}

In \cite{HL}, building on the model theory of algebraically closed valued fields ($\ACVF$), E. Hrushovski and F. Loeser developed a theory which provides a model-theoretic account of the Berkovich analytification of algebraic varieties. Informally, given a complete non-archimedean rank 1 valued field $k$ and an algebraic variety $X$ over $k$, they showed how the space of generically stable types on $X$ over a large algebraically closed valued field extending $k$, gives a model-theoretic avatar of the analytification $X^\mathrm{an}$ of $X$. Most notably, their association allowed them to obtain results concerning the homotopy type of quasi-projective varieties which were only known under strong algebro-geometric hypothesis on $X$ by results in \cite{Berkovich_contractible}.

One of the difficulties to study Berkovich spaces from a model-theoretic point of view is that such spaces do not seem to generally have (in $\ACVF$)  the structure of a definable set --where usual model-theoretic techniques can be applied-- but rather canonically the structure of a space of types. Part of the novelty of Hrushovski-Loeser's work lies on the fact that their spaces can be equipped with the structure of a (strict) pro-definable set, which granted them back the use of different classical model-theoretic tools. It is thus tempting to ask if such a structural result holds for other distinguished subsets of definable types and even for other first-order theories. It turned out this question is closely related to classical topics in model theory such as the model theory of pairs and uniform definability of types. In this article we give a positive answer in various contexts. Formally, we obtain the following result

\begin{theorem*}[Later Theorem \ref{thm:prodef}] Let $T$ be one of the following theories: 
\begin{enumerate}
    \item an complete o-minimal theory; 
    \item Presburger arithmetic; 
    \item the theory of a finite extension of $\Q_p$ ($\PCF_d$); 
    \item the theory of real closed valued fields $\RCVF$; 
    \item a completion of the theory of algebraically closed valued fields $\ACVF$;
    \item the theory of closed ordered differential fields $\CODF$.
\end{enumerate}
Then, definable types over models of $T$ are uniformly definable. In particular, for every model $M$ of $T$ and every $M$-definable set $X$, the space $\D{X}(M)$ of definable types concentrating on $X$ is pro-definable in $\cL^{\eq}$ (or any reduct in which $T$ admits elimination of imaginaries). 
\end{theorem*}

The fact that pro-definability follows from  uniform definability of types goes back to an argument of E. Hrushovski and F. Loeser in \cite{HL} which we present in Proposition \ref{prop:pro_def_eq}. In return, uniform definability of types is obtained from the following criterion which establishes a relation between this property and stably embedded pairs of models of $T$. 

\begin{theorem*}[Later Corollary \ref{thm:UD}] Let $T$ be an $\cL$-theory such that 
\begin{enumerate}
\item the theory of stably embedded pairs of models of $T$ is elementary in the language of pairs $\cL_P$; 
\item for every small model $M\models T$ and every definable type $p\in \D{x}(M)$, there is a stably embedded $(N,M)$ such that $P$ is realized in $N$.
\end{enumerate}
Then $T$ has uniform definability of types.  
\end{theorem*}

The previous criterion is derived from a slightly more general version (see Theorem \ref{thm:UD2}). 
As a corollary we also obtain pro-definability of some distinguished subspaces of the space of definable types, some of which have an interesting geometric interpretation. In particular, we aim to show that there are spaces of definable types in $\ACVF$ that can mimic Huber's analytification of an algebraic variety in a similar way the space of generically stable types mimics its analytification in the sense of Berkovich. We hope this can serve as a basis towards a model theory of adic spaces. In the same spirit, working in $\RCVF$, there are spaces of definable types which can be seen as the model-theoretic counterpart of the analytification of semi-algebraic sets as recently defined by P. Jell, C. Scheiderer and J. Yu in \cite{jell_etal}. This is also very much related to the work of C. Ealy, D. Haskell and J. Ma\v{r}\'ikov\'a in \cite{cli-has-mari} on residue field domination in real closed valued fields. This article aims to lay down a foundation for a model-theoretic study of such spaces. In a sequel, we will further explore structural properties of some of these spaces.

A natural question to ask is if such spaces are also strict pro-definable (i.e., pro-definable with surjective transition maps). Obtaining strictness is much more subtle and is often related to completions of theories of stably embedded pairs. Results towards a positive answer to this question will also be addressed in subsequent work. 

The article is laid out as follows. Section \ref{sec:prelim} sets up the notation and contains the needed model theoretic background. Spaces of definable types are introduced in Section \ref{sec:deftypes}. In Section \ref{sec:prodef}, we formally define what it means for spaces of definable types to be pro-definable and we show how this property is deduced from uniform definability of types (Proposition \ref{prop:pro_def_eq}. Section \ref{sec:SEP} focuses on stably embedded pairs and its main objective is to show that the class of such pairs is elementary for all theories above listed (see Theorem \ref{thm:definabletypes}. Section \ref{sec:UD} contains the proofs of the two theorems above and gathers a couple of open questions.  

\subsection*{Acknowledgements} We would like to thank: Sergei Starchenko and Kobi Peterzil for sharing with us the idea of using pairs as an approach to study spaces of definable types; Martin Hils for interesting discussions around the subject; Ludomir Newelski for sharing with us Example \ref{example:newelski}; and Elliot Kaplan for pointing us out a proof of Theorem \ref{thm:UDlist} for $\CODF$. 


P. Cubides Kovacsics was partially supported by the ERC project TOSSIBERG (Grant Agreement 637027), ERC project MOTMELSUM (Grant agreement 615722) and individual research grant
\emph{Archimedische und nicht-archimedische Stratifizierungen h\"oherer
Ordnung}, funded by the DFG. J. Ye was partially supported by NSF research grant DMS1500671.

\section{Preliminaries and notation}\label{sec:prelim}

\subsection{Model theoretic background}\label{sec:back_model}

We will use standard model-theoretic notation with the following specific conventions. 

Let $\cL$ be a first order language (possibly multi-sorted) and $T$ be a complete $\cL$-theory. The sorts of $\cL$ are denoted by bold letters $\sS$. Given a variable $x$, we let $\sS_x$ denote the sort where $x$ ranges. If $M$ is a model of $T$ and $\sS$ be a sort of $\cL$, we let $\sS(M)$ denote the set of elements of $M$ which are of sort $\sS$. Given an ordered tuple of variables $x=(x_i)_{i\in I}$ (possibly infinite), we extend this notation and set
\[
\sS_x(M)=\prod_{i\in I} \sS_{x_i}(M)
\]
By a subset of $M$ we mean a subset of , where $\sS$ is any sort in $\cL$. Let $C$ be a subset of $M$ (i.e. the union of all $\sS(M)$, $\sS$ a $\cL$-sort). The language $\cL(C)$ is the language $\cL$ together with constant symbols for every element in $C$. Given an $\cL(C)$-definable subset $X\subseteq \sS_x(M)$, we say that a type $p\in S_x(C)$ \emph{concentrates on $X$} if $p$ contains a formula defining $X$. We denote by $S_X(C)$ the subset of $S_x(C)$ consisting of those types concentrating on $X$. For a $C$-definable function $f\colon X\to Y$ the pushforward of $f$ is the function 
\[
f_*\colon S_X(C)\to S_Y(C) \hspace{1cm}  tp(a/C)\mapsto tp(f(a)/C). \] 
We let $\mathcal{U}$ be a monster model of $T$. As usual, a set is said to be \emph{small} if it is of cardinality smaller than $|\cU|$. A type $p(x)$ is a \emph{global type} if $p\in S_x(\cU)$. Following Shelah's terminology, a subset $X\subseteq \sS_x(M)$ is $*$-$C$-definable if there is a small collection $\Theta$ of $\cL(C)$-formulas $\varphi(x)$ (where only finitely many $x_i$ occur in each formula) such that $X=\{a\in \sS_x(M)\mid \cU\models \varphi(a), \varphi\in \Theta\}$. When the tuple of variables is finite, we also say $X$ is $\infty$-definable (or type-definable). 

We let $\dcl$ and $\acl$ denote the usual definable and algebraic closure model-thoeretic operators. Given a tuple of variables $x$ of length $n$, a $C$-definable set $X\subseteq \sS_x(\cU)$ and a subset $A$ of $\cU$, we let $X(A)\coloneqq X\cap \dcl(A)^n$.  

\subsection{Definable types}\label{sec:definabletype} We will mainly study definable types over models of $T$. Given a model $M$ of $T$ and a subset $A\subseteq M$, recall that a type $p\in S_x(M)$ is \emph{$A$-definable} (or definable over $A$) if for every $\cL$-formula $\varphi(x,y)$ there is an $\cL(A)$-formula $d_p(\varphi)(y)$ such that for every $c\in \sS_y(M)$
\begin{equation}\label{eq:defscheme}\stepcounter{eqn}\tag{E\arabic{eqn}}
\varphi(x,c)\in p(x) \Leftrightarrow \ M\models d_p(\varphi)(c). 
\end{equation}
The map $\varphi(x,y)\mapsto d_p(\varphi)(y)$ is called a \emph{scheme of definition for $p$}, and the formula $d_p(\varphi)(y)$ is called a \emph{$\varphi$-definition for $p$}. We say $p\in S_x(M)$ is \emph{definable} if it is $M$-definable. Given any set $B$ containing $M$, there is a canonical extension of $p$ to $S_x(B)$, denoted by $p|B$ and given by
\[
p|B\coloneqq \{\varphi(x,b)\mid N\models d_p(\varphi)(b)\},  
\]
where $N$ is any model of $T$ containing $B$. We refer the reader to \cite[Section 1]{pillay83} for proofs and details of these facts. The following lemma is also folklore. 

\begin{lemma}\label{lem:general_def_types} If $tp(a_1/M)$ is definable and $a_2\in \acl(Ma_1)$, then $tp(a_2/M)$ is definable. \qed
\end{lemma}

Let $\varphi(x;y)$ be a partitioned formula. A formula $\psi(y,z_\varphi)$ is a \emph{uniform definition for $\varphi$ (in $T$)} if for every model $M$ of $T$ and every $p\in \D{x}(M)$ there is $c=c(p,\varphi)\in \sS_{z_\varphi}(M)$ such that $\psi(y,c)$ is a $\varphi$-definition for $p$. We say $T$ has \emph{uniform definability of types} if every partitioned $\cL$-formula $\varphi(x;y)$ has a uniform definition, which we write $d(\varphi)(y,z_\varphi)$. The following is a routine coding exercise. 

\begin{lemma}\label{lem:UD_finite} Let $\varphi(x;y)$ be a partitioned $\cL$-formula. Suppose there are finitely many $\cL$-formulas $\psi_1(y,z_1), \ldots, \psi_n(y,z_n)$ such that for every model $M$ of $T$ and every $p\in \D{x}(M)$, there are $i\in\{1,\ldots,n\}$ and $c\in \sS_{z_i}(M)$ such that $\psi_i(y,c)$ is a $\varphi$-definition for $p$. Then $\varphi(x;y)$ has a uniform definition. 
\qed
\end{lemma}

\section{Completions by definable types}\label{sec:deftypes} 

In this section we define various ``completions'' of definable sets, using definable types. The main one consists precisely of the set of all definable types which concentrate on a given definable set. 

\subsection{The definable completion} \label{sec:definable_completion}
Let $X\subseteq \sS_x(\cU)$ be a $C$-definable set and $A$ be a small set. The \emph{definable completion of $X$ over $A$}, denoted $\D{X}(A)$, is the space of $A$-definable global types which concentrate on $X$. For a tuple of variables $x$ we write $\D{x}(A)$ for $\D{\sS_x(\cU)}(A)$. 

\begin{remark}\label{rem:not-clash}
Note that there is a small notation clash with respect to parenthesis: $\D{X}(A)$ is \emph{not} a subset of $S_X(A)$ (the former is contained in $S_X(\cU)$). However, when $A$ is a model and $X$ is $A$-definable, the restriction map $S_X(\cU)\to S_X(A)$ defines a one-to-one correspondence between $\D{X}(A)$ and the set of definable types in $S_X(A)$. The inverse of such a map is given by extension of definable types over models as above defined.   
\end{remark}


It is an easy exercise to check that the pushforward of a definable type is again definable. We will use the following functorial notation: 

\begin{lemma}\label{lem:can_extension}
	 Let $X$ and $Y$ be $C$-definable sets and $f\colon X\to Y$ be a $C$-definable map. Let $f^\Deff$ denote the restriction to $\D{X}(A)$ of the pushforward map $f_*$. Then $f^\Deff$ defines a map 
	 \[
	 f^{\Deff}\colon\D{X}(A)\to\D{Y}(A)
	 \]
for every small set $A$ containing $C$. \qed
\end{lemma}
 

\begin{remark}\label{rem:surjectivity_transfer} \hspace{1cm}
\begin{enumerate}[leftmargin=*]
    \item Note that if $f$ is injective, so is $f_
    *$ and hence $f^\Deff$. However, the transfer of surjectivity from $f$ to $f^\Deff$ is more subtle. We will say that $T$ has \emph{surjectivity transfer} precisely if for every surjective definable function $f\colon X\to Y$, the function $f^\Deff$ is surjective. Note that if $T$ has definable Skolem functions, then it has surjectivity transfer. Indeed, let $g\colon Y\to X$ be a definable section of $f$. Then, given any type $p\in \D{Y}$, we have that $f^\Deff(g^\Deff(p))=p$, so $f^\Deff$ is surjective. We do not know whether $\CODF$ has surjectivity transfer. 

\item By \cite[Lemma 4.2.6]{HL}, a stronger results holds for o-minimal theories and $\ACVF$. Indeed, the function $f^\Deff$ is surjective even when $X$ and $Y$ are definable subsets of products of imaginary sorts. Adapting the argument given in \cite[Lemma 4.2.6(a)]{HL}, the same result holds for $\RCVF$. 

\end{enumerate}
\end{remark}



 


\subsection{Other completions by definable types}

\subsubsection{Bounded types}\label{sec:bd_ort}
Let $T$ be an o-minimal theory and $M$ be a model of $T$. Given an elementary extension $M\preceq N$, we say that \emph{$N$ is bounded by $M$} if for every $b\in N$, there are $c_1, c_2\in M$ such that $c_1\leqslant b\leqslant c_2$. Let $A$ be a small subset of $\cU$ and $X$ be a definable set. A type $p\in \D{X}(A)$ is \emph{bounded} if for any small model $M$ containing $A$ and every realization $a\models p|M$, there is an elementary extension $M\preceq N$ with $a\in X(N)$ and such that $N$ is bounded by $M$. 

Let $T$ be either $\RCVF$ or a completion of $\ACVF$. Let $A$ be a small subset of $\cU$. A type $p\in \D{X}(A)$ is \emph{bounded} if for any small model $M$ containing $A$ and every realization $a\models p|M$, $\Gamma(\acl(Ma))$ is bounded by $\Gamma(M)$, where $\Gamma$ denotes the value group sort.

\subsubsection{The bounded completion}\label{def:bounded} 
Let $T$ be either an o-minimal theory, a completion of $\ACVF$ or $\RCVF$. Let $A$ be a small subset of $\cU$ and $X$ be a definable set. The \emph{bounded completion of $X$ over $A$}, denoted $\widetilde{X}(A)$, is the set of bounded global $A$-definable types. 

\subsubsection{Types orthogonal to $\Gamma$}\label{sec:orth_gamma}
Let $T$ be a complete theory extending the theory of valued fields (\emph{e.g.}, $\RCVF$ or a completion of $\ACVF$). Let $A$ be a small subset of $\cU$. A type $p\in \D{x}(A)$ is said to be \emph{orthogonal to $\Gamma$} if for every model $M$ containing $A$ and every realization $a\models p|M$, $\Gamma(M)=\Gamma(\acl(Ma))$.

\subsubsection{The orthogonal completion}\label{sec:stable_completion} Let $T$ be either a completion of $\ACVF$ or $\RCVF$. Given a definable set $X$, the \emph{orthogonal completion of $X$}, denoted by $\widehat{X}(A)$, is the set of global $A$-definable types concentrating on $X$ which are orthogonal to $\Gamma$. 

By \cite[Proposition 2.9.1]{HL}, when $T$ is a completion of $\ACVF$, the set $\widehat{X}(A)$ also corresponds to the set of definable types which are generically stable and is also called the \emph{stable completion of $X$}. Note that this equivalence does not hold in $\RCVF$ as in this theory every generically stable type is a realized type.    

\begin{remark}\label{rem:bounded}\
\begin{enumerate}[leftmargin=*]
    \item If $T$ is an o-minimal expansion of the theory of real closed field, then every bounded definable type is a realized type. However, this is not the case for general o-minimal theories. For example, the type of an element arbitrarily close to zero in $\DOAG$ is bounded.
    \item Let $T$ be one of the theories listed in \ref{def:bounded} and $M$ be a model of $T$. If $\tp(a/M)$ is bounded, then $tp(b/M)$ is bounded for every $b \in \acl(M,a)$. 
    \item Let $T$ be either a completion of $\ACVF$ or $\RCVF$, and let $M$ a model of $T$. If $\tp(a/M)$ is orthogonal to $\Gamma$, then $tp(b/M)$ is orthogonal to $\Gamma$ for every $b \in \acl(M,a)$.
    \item Given a definable function $f\colon X\to Y$, then the pushforward $f*$ of $f$ restricted to $\widetilde{X}$ (resp. to $\widehat{X}$) has image in $\widetilde{Y}$ (resp. $\widehat{Y}$). Similarly as for the definable completion, one uses the more functorial notation $\widetilde{f}$ (resp. $\widehat{f}$) to  denote the restriction of $f_*$ to $\widetilde{X}$ (resp. $\widehat{X}$). 
\end{enumerate}
\end{remark}

\subsubsection{Geometric interpretation} For $T$ a completion of $\ACVF$, let $V$ be a variety over a complete rank 1 valued field $F$. In \cite{HL}, $\widehat{V}$ is introduced as a model-theoretic analogue of the Berkovich analytification $V^{\mathrm{an}}$ of $V$. Similarly, our aim is to view $\widetilde{V}$ as a model-theoretic analogue of the Huber analytification of $V$. When $T$ is $\RCVF$, $\widehat{V}$ is a good candidate to be the model-theoretic counterpart of the analytification of semi-algebraic sets defined by Jell, Scheiderer and Yu in \cite{jell_etal}. The set $\widehat{V}$ also corresponds to the set of residue field dominated types as defined by Ealy, Haskell and Ma\v{r}\'iíkov\'a in \cite{cli-has-mari}. The space $\widetilde{V}$ (in $\RCVF$) seems to suggest there is an analogue of Huber's analytification of semi-algebraic sets. Finally, $\D{V}$ can be viewed as a model-theoretic analogue of the ``space of valuations on $V$". As mentioned in the introduction, we will present more structural results concerning these spaces in a sequel of this article.

\section{Spaces of definable types as pro-definable sets}

\subsection{Pro-definability}\label{sec:prodef}

\subsubsection{Pro-definable sets}\label{sec:prodef_sets}
Let $(I,\leq)$ be a small upwards directed partially ordered set and $C$ be a small subset of $\cU$. A \emph{$C$-definable projective system} is a collection $(X_i, f_{ij})$ such that 
\begin{enumerate}
\item for every $i\in I$, $X_i$ is a $C$-definable set; 
\item for every $i, j\in I$ such that $i\geqslant j$; $f_{ij}\colon X_i\to X_j$ is $C$-definable;
\item $f_{ii}$ is the identity on $X_i$ and $f_{ik} = f_{jk}\circ f_{ij}$ for all $i \geqslant j \geqslant k$.
\end{enumerate}  

A \emph{pro-$C$-definable set} $X$ is the projective limit of a $C$-definable projective system $(X_i, f_{ij})$
\[
X\coloneqq \varprojlim_{i\in I} X_i.
\]
We say that $X$ is \emph{pro-definable} if it is pro-$C$-definable for some small set of parameters $C$. Pro-definable sets can also be seen as $\ast$-definable sets. By a result of Kamensky \cite{kamensky}, we may identify $X$ and $X(\cU)$.  

\subsubsection{Pro-definable morphisms}\label{sec:prodef_mor}

Let $X=\varprojlim_{i\in I} X_i$ and $Y=\varprojlim_{j\in J} Y_j$ be two pro-$C$-definable sets with associated $C$-definable projective systems $(X_i,f_{ii'})$ and $(Y_j, g_{jj'})$. A \emph{pro-$C$-definable morphism} is the data of a monotone function $d\colon J \to I$ and a family of $C$-definable functions $\{\varphi_{ij}\colon X_i\to Y_j \mid i\geqslant d(j)\}$ such that, for all $j\geqslant j'$ in $J$ and all $i \geqslant i'$ in $I$ with $i\geqslant d(j)$ and $i'\geqslant d(j')$, the following diagram commutes 
\[
\begin{tikzcd}
X_{i} \ar{r}{f_{ii'}} \ar{d}{\varphi_{ij}} &  X_{i'} \ar{d}{\varphi_{i'j'}} \\
Y_{j} \ar{r}{g_{jj'}} & Y_{j'}.
\end{tikzcd}
\]
By \cite{kamensky}, taking $\cU$-points, the above data defines a map $\varphi\colon X(\cU)\to Y(\cU)$ satisfying the following commutative diagram
\[
\begin{tikzcd}
X \ar{d}{\varphi} \ar{r}{\pi_i} & X_{i} \ar{r}{f_{ii'}} \ar{d}{\varphi_{ij}} &  X_{i'} \ar{d}{\varphi_{i'j'}} \\
Y \ar{r}{\pi_j}  & Y_{j} \ar{r}{g_{jj'}} & Y_{j'},
\end{tikzcd}
\]
where $\pi_i$ and $\pi_j$ denote the canonical projections $\pi_i\colon X\to X_i$ and $\pi_j\colon Y\to Y_j$. Two pro-definable maps $(\varphi_{ij}, d)$ and $(\varphi_{ij}', d')$ are identified if they induced the same map from $X(\cU)$ to $Y(\cU)$.

\subsection{Completions by definable types as pro-definable sets}\label{sec:type_pro} 

As stated in the introduction, we would like to equip spaces of definable types with a pro-definable structure. Let us explain the precise meaning of this in the case of the definable completion. The other completions are handled analogously. 

We say that \emph{definable types are pro-definable in $T$} if for every model $M$ of $T$, for every $M$-definable set $X$, for $M$-definable function $f\colon X\to Y$ and every elementary extension $N$, there are a pro-$M$-definable set $P_X(M)$, a bijection $h_X^N\colon S^\Def_X(N)\to P_X(N)$ and a pro-$M$-definable morphim $f'_N\colon P_X(N)\to P_Y(N)$ making the following diagram commute
\[
\begin{tikzcd}
&
S_Y^\Def(N)
  \arrow[rr, "h_Y^N"]
&&
P_Y(N)
\\
S_X^\Def(N) 
  \arrow[rr, crossing over, "h_X^N" near end]
  \arrow[ur,"f^\Def"]  
&&
P_X(N)
\arrow[ur,"f_N'"]  
\\
&
S_Y^\Def(M)
  \arrow[rr, "h_Y^M" near start]
  \arrow[uu, crossing over]
&&
P_Y(M)
  \arrow[uu,"i"]
\\
S_X^\Def(M)
  \arrow[rr, "h_X^M"] 
  \arrow[uu] 
  \arrow[ur,"f^\Def"]  
&&
P_X(M)
  \arrow[uu, crossing over, "i" near end]
    \arrow[ur,swap,"f_M'"]  
\end{tikzcd}
\]

\noindent Note that the map $f_{\bullet}'$ is completely determined by the commutativity of the diagram. The condition we impose states that such a map has to be a pro-$M$-morphism. Similarly, we say a subfunctor $\mathcal{C}(\bullet)$ of $\D{\bullet}$ is pro-definable if $\D{\bullet}$ is the above diagram is replaced by $\mathcal{C}(\bullet)$. It is in this sense that we say that the bounded and orthogonal completions are pro-definable (in both cases we have a subfunctor by Remark \ref{rem:bounded}). 

\

The following result, essentially due to E. Hrushovski and F. Loeser \cite[Lemma 2.5.1]{HL}, shows the link between uniform definability of types and pro-definability. We include a proof for the reader's convenience.  

\begin{proposition}\label{prop:pro_def_eq} Suppose $T$ has uniform definability of types. Then definable types are pro-definable in $\cL^{\eq}$. In particular, if $T$ has elimination of imaginaries, then definable types are pro-definable in $T$. \qed
\end{proposition} 

\begin{proof} Fix some model $M$ of $T$. Given a partitioned $\cL$-formula $\varphi(x;y)$, let $d(\varphi)(y,z_\varphi)$ be a uniform definition for $\varphi$. By possibly using an $\cL^{eq}$-formula, we may suppose that $z_\varphi$ is a single variable and 
\[
T^\eq\models (\forall z_\varphi)(\forall z_\varphi')(\forall y) [(d(\varphi)(y,z_\varphi)\leftrightarrow d(\varphi)(y,z_\varphi')) \to z_\varphi=z_\varphi'].
\]
Let $\Phi$ denote the set of partitioned $\cL$-formulas of the form $\varphi(x;y)$ where $y$ ranges over all finite tuples of variables. We associate to this data a map $\tau$ defined by 
\begin{equation}\label{eq:pi_d}\stepcounter{eqn}\tag{E\arabic{eqn}}
\tau\colon\D{x}(M)\to \prod_{\varphi\in \Phi} \sS_{z_\varphi}(M) \hspace{1cm} p\mapsto (c(p,\varphi))_{\varphi\in \Phi}, 
\end{equation}
where $c(p,\varphi)$ is such that $d(\varphi)(y,c(p,\varphi))$ is a $\varphi$-definition for $p$. The codomain of $\tau$ is a pro-definable set (being a small product of definable sets). Thus, since $\tau$ is injective, to obtain pro-definability it suffices to show that $\tau(\D{x}(M))$ is $*$-definable. 

Without loss of generality we may suppose that the map $d$ factors through Boolean combinations, that is, $d(\varphi\wedge \psi)=d(\varphi)\wedge d(\psi)$ and $d(\neg\varphi)=\neg d(\varphi)$.  
Consider the following set of formulas $\Theta$ containing, for each $\cL$-formula $\varphi(x;y)\in \Phi$, the formula $\theta_{\varphi}(z_{\varphi})$ given by  
\[
\theta_{\varphi}(z_{\varphi})\coloneqq   (\forall y) (\exists x) (\varphi(x,y)\leftrightarrow d(\varphi)(y,z_\varphi)). 
\]
Given~$\varphi_1(x;y_1), \ldots, \varphi_m(x;y_m)$ in $\Phi$, let $y$ be the tuple $(y_1,\ldots,y_m)$ and $\varphi(x;y)$ denote the conjunction $\bigwedge_{i=1}^m \varphi_i(x,y_i)$. Since $d$ factors through conjunctions 
\begin{equation}\label{eq:conjunction} \stepcounter{eqn}\tag{E\arabic{eqn}}
\models (\forall z_\varphi) (\forall z_{\varphi_1}) \cdots (\forall z_{\varphi_m}) (\forall y) (d(\varphi)(y,z_\varphi) \leftrightarrow \bigwedge_{i=1}^m d(\varphi_i)(y,z_{\varphi_i})). 
\end{equation}
We claim that 
\[
\tau(\D{X}(M))=\{(c_{\varphi})_{\varphi} \in \prod_{\varphi\in\Phi} \sS_\varphi \mid (c_\varphi)_\varphi \models \Theta\}. 
\] 
From left-to-right, let $p\in \D{x}$ and $\theta_\varphi(z_\varphi)$ be a formula in $\Theta$. We have that   
\begin{align*}
\tau(p)\models \theta_\varphi(z_\varphi)  	& \ \Leftrightarrow (\cc(p,\varphi))_{\varphi}\models \theta_\varphi(z_\varphi)  \\
															& \ \Leftrightarrow  \cU \models (\forall y) (\exists x) (\varphi(x,y)\leftrightarrow d(\varphi)(y,\cc(p,\varphi))),  
\end{align*}
and the last formula holds since for every $y$ any realization of $p$ satisfies such a formula. 

To show the right-to-left inclusion, let $(c_\varphi)_\varphi$ be such that $(c_\varphi)_\varphi\models\Theta$. Consider the set of formulas
\[
p(x)\coloneqq\{\varphi(x,b) \mid \cU\models d(\varphi)(b,c_\varphi)\}. 
\]
Let us show that $p(x)$ is an element of $S_x(\cU)$. Once we show $p(x)$ is consistent, that $p\in\D{x}(M)$ follows by definition. Let $\varphi_1(x,b_1),\ldots, \varphi_m(x,b_m)$ be formulas in $p(x)$. Letting $b\coloneqq (b_1,\ldots,b_m)$, by the definition of $p(x)$ and \eqref{eq:conjunction} we have that the formula $\varphi(x,b)$ is also in $p(x)$. Moreover, since $(c_\varphi)_\varphi\models \theta_{\varphi}$ we have in particular that 
\[
\models (\exists x)(\varphi(x,b)\leftrightarrow d(\varphi)(b, c_\varphi)).  
\]
Finally, since $\varphi(x,b)\in p(x)$, we must have that $\models d(\varphi)(b,c_\varphi)$, which shows there is an element satisfying $\varphi(x,b)$. By compactness, $p(x)$ is consistent. 

Given an $M$-definable subset $X\subseteq \sS_x(M)$, we endow $\D{X}(M)$ with the pro-definable structure inherited from $\D{x}(M)$. More precisely, if $X$ is defined by a formula $\psi(x,a)$ for some tuple $a \in M$, we have that  
\[
\tau(\D{X}(M)) = \{(c_{\varphi})_{\varphi} \in \prod_{\varphi\in\Phi} \sS_\varphi \mid (c_\varphi)_\varphi \models \Theta \cup \{d(\psi)(a,z_\psi)\}\}. 
\]

We leave as an exercise to show that the present construction guarantees all the above functoriality properties.
\end{proof}

\begin{corollary}\label{cor:pro_def_eq} Suppose $T$ has uniform definability of types. Let $M$ be a model and $X$ be a definable subset of $M$. Then every $\ast$-definable subset of $\D{X}(M)$ is pro-definable. \qed
\end{corollary}

\begin{remark}
If $T$ has uniform definability of types and has surjectivity transfer for definable functions in $\cL^{\eq}$ (see (2) of Remark \ref{rem:surjectivity_transfer}), then definable types concentrating on a product of imaginary sorts are also uniformly definable. In particular, by Proposition \ref{prop:pro_def_eq}, if $X$ is a definable subset of some product of imaginary sorts, then $\D{X}(M)$ is pro-definable (in $\cL^{\eq}$). 

\end{remark}

\section{Stably embedded pairs and elementarity}\label{sec:SEP}

Suppose $\cL$ is a one-sorted language. Let $\cL_P$ be a language extending $\cL$ by a new unary predicate $P$. We denote an $\cL_P$-structure as a pair $(N, A)$ where $N$ is an $\cL$-structure and $A\subseteq N$ corresponds to the interpretation of $P$. Given a complete $\cL$-theory $T$, the $\cL_P$-theory of \emph{elementary pairs of models of $T$}, denoted $T_P$. Given a tuple $x=(x_1,\ldots,x_m)$, we abuse of notation and write $P(x)$ as an abbreviation for $\bigwedge_{i=1}^n P(x_i)$. When $\cL$ is multi-sorted we let $\cL_P$ denote the language which extends $\cL$ by a new unary predicate $P_\sS$ for every $\cL$-sort $\sS$. Analogously, an $\cL_P$-structure $N$ is a model of $T_P$ if the collection of subsets $P_\sS(N)$ forms an elementary $\cL$-substructure of $N$. We will also denote any such a structure as a pair $(N, M)$ where $M\prec N\models T$ and for every $\cL$-sort $\sS$, $P_\sS(N)=\sS(M)$.  


\subsection{Stable embedded pairs}\label{sec:SEpairs1} Let $M\prec N$ be an elementary extension of models of $T$. The extension is called \emph{stably embedded} if for every $\cL(N)$-definable subset $X\subseteq \sS_x(N)$, the set $X\cap \sS_x(M)$ is $\cL(M)$-definable.

The class of \emph{stably embedded models of $T$}, denoted $\SE(T)$, is the class of $\cL_P$-structures $(N, M)$ such that $M\prec N\models T$ and the extension $M\prec N$ is stably embedded. The following lemma is a standard exercise. 

\begin{lemma}\label{lem:stably_def} An elementary extension $M\prec N$ is stably embedded if and only if for every tuple $a$ in $N$, the type $tp(a/M)$ is definable. \qed
\end{lemma}

The main objective of this section is to show that for various NIP theories, the class of stably embedded pairs is an $\cL_P$-elementary class. Stable theories constitute a trivial example of this phenomenon, since the class of stably embedded pairs coincides with the class of elementary pairs. O-minimal theories constitute a less trivial example. Let us first recall Marker-Steinhorn characterizarion of definable types in o-minimal structures. 

\begin{theorem}[{\cite[Theorem 2.1]{MS}}]\label{thm:MS} Let $T$ be an o-minimal theory and $M$ be a model of $T$. Then, $p\in S_x(M)$ is definable if and only if for every realization $a$ of $p$, $M$ is Dedekind complete in $M(a)$ (where $M(a)$ is the prime model of $M$ over $a$). 
\end{theorem}

Since being Dedekind complete is expressible in $\cL_P$, we readily obtain:

\begin{corollary}\label{cor:omin_elementary} Let $T$ be a complete o-minimal theory. Then the class of stably embedded pairs is $\cL_P$-elementary. 
\end{corollary}

The following result of Q. Brouette shows the analogue result for $\CODF$.  

\begin{theorem}[{\cite[Proposition 3.6]{brouette}}] Let $M$ be a model of $\CODF$. Then, $p\in S_x(M)$ is definable if and only if for every realization $a$ of $p$, $M$ is Dedekind complete in the real closure of the ordered differential field generated by $M\cup \{a\}$.  
\end{theorem}

\begin{corollary}\label{cor:CODF_elementary} Then class of stably embedded pairs of models of $\CODF$ is $\cL_P$-elementary. 
\end{corollary}



A similar situation holds in Presburger arithmetic. The following characterization of definable types over models in $T$ is due to G. Conant and S. Vojdani (see \cite{f-generic_pres}). 

\begin{theorem}\label{thm:pres1} Let $M$ be a model of Presburger arithmetic. A type $p\in S_n(M)$ is definable if and only if for every realization $a\models p$, $M\prec M(a)$ is an end-extension.  \qed
\end{theorem}  

 \begin{corollary}\label{cor:pres} An elementary pair $(N, M)$ of models of Presburger arithmetic is stably embedded if and only if it is an end-extension. In particular, the class $\SE(T)$ is an elementary class in $\cL_P$. \qed
\end{corollary}  

\subsection{Stable embedded pairs of valued fields}

In what follows we gather the corresponding characterization of stably embedded pairs of models for $\ACVF$, $\RCVF$ and $\PCF_d$. We need first the following notation and terminology for induced structures. Let $M$ be an $\cL$-structure and $\sS$ be an imaginary sort in $M^{\eq}$. We let $\cL_\sS$ be the language having a predicate $P_R$ for every $\cL$-definable (without parameters) subset $R\subseteq \sS^n(M)$. The structure $(\sS(M),\cL_\sS)$ in which every $P_R$ is interpreted as the set $R$ is called the \emph{induced structure on $\sS(M)$}. The sort $\sS$ is called \emph{stably embedded} if every $\cL^\eq$-definable subset of $\sS(M)$ is $\cL_\sS$-definable. The following lemma is left to the reader.  

\begin{lemma}\label{lem:eq} Let $T$ be an complete $\cL$-theory and $(N,M)$ be a stably embedded pair of models of $T$. Let $\sS$ be an $\cL^{\eq}$-sort which is stably embedded (as a sort) in every model of $T$. Then the pair $\sS(M)\prec \sS(N)$ is stably embedded in $\cL_\sS$.  
\end{lemma}

For a valued field $(K,v)$ we let $\Gamma_K$ denote the value group, $\mathcal{O}_K$ its valuation ring, $k_K$ the residue field and $\res\colon \mathcal{O}_K\to k_K$ the residue map. Given a valued field extension $(K\subseteq L, v)$ and a subset $A\coloneqq \{a_1,\ldots,a_n\}\subseteq L$, we say that $A$ is \emph{$K$-valuation independent} if for every $K$-linear combination $\sum_{i=1}^n c_ia_i$ with $c_i\in K$, 
\[
v\left(\sum_{i=1}^n c_ia_i\right) = \min_{i}(v(c_ia_i)). 
\]
The extension $L|K$ is called \emph{$vs$-defectless}\footnote{This is the same as separated as in \cite{baur} and \cite{delon-sep}} if every finitely generated $K$-vector subspace $V$ of $L$ admits a \emph{$K$-valuation basis}, that is, a $K$-valuation independent set which spans $V$ over $K$. 

\begin{remark}\label{rem:stablyemb}
By standard arguments, both the value group and the residue field are stably embedded sorts in $\ACVF$, $\RCVF$ and $\PCF_d$. As a corollary of Lemma \ref{lem:eq} we obtain that if $(K\prec L,v)$ is a stably embedded pair of models of any of these theories, then, the pairs $\Gamma_K\prec \Gamma_L$ and $k_K\prec k_L$ are stably embedded in their respective induced structure languages. 
\end{remark}

Let us now explain how to show that the class $\SE(T)$ is $\cL_P$-elementary for $T$ either $\ACVF$, $\RCVF$ or $\PCF_d$. In all three cases, the result will follow from the following theorem:  

\begin{theorem}\label{thm:valued-charac} Let $T$ be either $\ACVF$, $\RCVF$ or $\PCF_d$. Let $K\prec L$ be a pair of models of $T$. Then the following are equivalent
\begin{enumerate}
\item the pair $K\prec L$ is stably embedded 
\item the valued field extension $L|K$ is $vs$-defectless, the pairs $\Gamma_K\prec\Gamma_L$ and $k_K\prec k_L$ are stably embedded.  
\end{enumerate}
\end{theorem}

For $\ACVF$ the above Theorem is precisely the content of \cite[Theorem 1.9]{CD}. For $\RCVF$ and $\PCF_d$ the result is new and the corresponding proof is presented in Sections \ref{sec:appendix_RCVF} and \ref{sec:appendix_PCF}. Note that both in $\ACVF$ and $\PCF_d$ the condition on the residue field extension can be removed. Indeed, in $\ACVF$, the induced structure on the residue field is that of a pure algebraically closed field, and therefore (since such a structure is stable) the corresponding extension $k_K\prec k_L$ is always stably embedded. In $\PCF_d$, the extension $k_K\prec k_L$ is finite and hence trivially stably embedded.   

\begin{corollary}\label{cor:SE-valued}
Let $T$ be either $T$ either $\ACVF$, $\RCVF$ or $\PCF_d$. Then the class $\SE(T)$ is elementary in the language of pairs. 
\end{corollary} 

\begin{proof}
By Theorem \ref{thm:valued-charac}, it suffices to show that the condition stated in part (2) is elementary in the language of pairs. Note that being a $vs$-defectless extension is an elementary property in the language of pairs, so we only need to show that having stably embedded value group and residue field extensions is an elementary property in the language of pairs. 

Concerning the value group, both in $\ACVF$ and $\RCVF$, the value groups extension is an elementary extension of models of an o-minimal theory (namely, divisible ordered abelian groups). Therefore, being stably embedded is elementary by Corollary \ref{cor:omin_elementary} and Remark \ref{rem:stablyemb}. For $\PCF_d$, the value group extension is an extension of models of Presburger arithmetic, and hence the result follows by Corollary \ref{cor:pres}. 

Regarding the residue field extension, as explained above, it only plays a role for $\RCVF$. In this case, the residue field extension corresponds to an elementary pair of real closed fields. As before, the result follows by Corollary \ref{thm:definabletypes} and Remark \ref{rem:stablyemb}. 
\end{proof}

It is worthy to acknowledge that recently, building on the present results, P. Touchard extend some of these results to other classes of henselian fields \cite{touchard-pair}.

\subsubsection{Stably embedded pairs of real closed valued fields}\label{sec:appendix_RCVF}

The study real closed valued fields in the language $\cL_{\mathrm{div}}^{\leqslant}$ which corresponds to the language of ordered rings extended by a binary predicate $\Div(x,y)$ which holds if and only if $v(x)\leqslant v(y)$. The $\cL_{\mathrm{div}}^{\leqslant}$-theory of real closed valued fields $\RCVF$ has quantifier elimination (see \cite{cherlin_dickmann}). 

\begin{proof}[Proof of Theorem \ref{thm:valued-charac} for $\RCVF$:] Let $(K\prec L,v)$ be a pair of real closed valued fields. 

$(1)\Rightarrow (2)$: By Remark \ref{rem:stablyemb}, the pairs $\Gamma_K\prec \Gamma_L$ and $k_K\prec k_L$ are stably embedded. That the extension is $vs$-defectless follows word for word as for algebraically closed valued fields in \cite[Theorem 1.9]{CD}. 

$(2)\Rightarrow (1)$: Let $X\subseteq L^m$ be an $\cL_{\mathrm{div}}^{\leqslant}$-definable set over $L$. We need to show that $X\cap K^m$ is $\cL_{\mathrm{div}}^{\leqslant}$-definable over $K$. By quantifier elimination, we may suppose that $X$ is defined by one of the following formulas 
\begin{enumerate}[$(i)$]
\item $v(P(x))\square v(Q(x))$ with $\square$ either $\leqslant$ or $<$, 
\item $0<P(x)$,
\end{enumerate}
where $P, Q\in L[X]$ with $X=(X_1,\ldots, X_m)$. When $X$ is defined by a formula as in $(i)$, one can proceed as in \cite[Theorem 1.9]{CD}, so it remains to show the result for $(ii)$. 

Since the extension $L|K$ is $vs$-defectless, there is a $K$-valuation independent set~$A\coloneqq \{a_1,\ldots,a_n\}\subseteq L$ and polynomials~$P_i\in K[X]$ with~$i\in I\coloneqq \{1,\ldots, n\}$ such that~$P(x)=\sum_{i\in I} a_iP_i(x)$. By \cite[Lemma 2.24]{bla-cubi-kuh}, we can further suppose that for every $i,j\in I$, if  $v(a_i)$ and $v(a_j)$ lie in the same coset modulo $\Gamma_K$, then $v(a_i) = v(a_j)$. Moreover, at the expense of multiplying $P_i$ by $-1$, we can suppose that $a_i>0$ for all $i\in I$. 

For each $\emptyset\neq J\subseteq I$, let $A_J$ be the set 
\[
A_J\coloneqq \{x\in K^m\mid v(\sum_{i\in I}a_iP_i(x)) = v(a_jP_j(x)) \text{ if and only if } j\in J\}.  
\]
By case $(i)$, we may suppose $A_J$ is definable over $K$. Further, since $A$ is $K$-valuation independent, the sets $A_J$ cover $K^m$ when $J$ varies over all non-empty subsets of $I$. Therefore, it suffices to show that $X\cap A_J$ is definable over $K$ for every $J\subseteq I$. Let us first show how to reduce to the case where $J=I$. If $J\neq I$, then for all $x\in A_J$ we have 
\[
0<P(x) \Leftrightarrow 0< \sum_{i\in J} a_iP_i(x) + \sum_{i\in I\setminus J} a_iP_i(x) \Leftrightarrow 0< \sum_{i\in J} a_iP_i(x),   
\]
and thus we obtain an equivalent formula where for all $i\in J$, $v(a_iP_i(x))$ is the same. Therefore without loss of generality it suffices to show the case $J=I$. Now, since for all $x\in A_I$, $v(a_iP_i(x))=v(a_jP_j(x))$ for all $i,j\in I$, $v(a_i)$ and $v(a_j)$ are in the same coset modulo $\Gamma_K$, and hence $v(a_i)=v(a_j)$ for all $i,j\in I$. Also $v(P_i(x))=v(P_j(x))$ for all $i,j\in I$. Multiplying by a suitable constant $c\in L$, we may suppose that $v(a_i)=0$ for all $i\in I$. Similarly, multiplying by a suitable constant $c'\in K$, we may suppose that $v(P_i(x))=0$ for all $x\in A_I$. We conclude by noting that in this situation, for all $x\in A_I$   
\[
0<P(x) \Leftrightarrow 0<\res(\sum_{i=1}^n a_i P_i(x))  \Leftrightarrow 0<\sum_{i=1}^n \res(a_i)\res(P_i(x)). 
\]
Since $k_K$ is stably embedded in $k_L$, the set $\{y\in k_K^m\mid 0<\sum_{i=1}^n \res(a_i) y_i\}$ is definable over $k_K$. Lifting the parameters, we obtain that $X\cap A_I$ is definable over $K$. 
\end{proof}

\subsubsection{Stably embedded pairs of models of $\PCF_d$}\label{sec:appendix_PCF}

Let $K$ be a finite extension of $\Q_p$. Let $d$ be the $p$-rank of the extension (see \cite[Section 2]{prestel2005}). The theory $\PCF_d$ is the theory of $K$ in the language $\cL_{\mathrm{Mac}}$ with finitely many new constant symbols as introduced in \cite[Theorem 5.6]{prestel2005}. Such a theory admits elimination of quantifiers. 

We need some preliminary lemmas. 

\begin{lemma}\label{lem:stpart} Let $(K\subseteq L, v)$ be a valued field extension. Suppose that every $y\in L$ is of the form $y=x+a$ with $a\in K$ and $x\in L$ such that $|v(x)|>\Gamma_K$. Then the extension is $vs$-defectless. 
\end{lemma} 

\begin{proof}
Let $V\subseteq L$ be a $K$-vector space of dimension $n$. Let us show that $V$ contains elements $\{x_1,\ldots, x_n\}$ such that each $v(x_i)$ lies in a different $\Gamma_K$-coset. By \cite[Lemma 3.2.2]{prestel2005}, this implies that $\{x_1,\ldots, x_n\}$ is a $K$-valuation basis for $V$. We proceed by induction on $n$. Let $\{y_1,\ldots, y_n\}$ be a basis for $V$. For $n=1$ the result is trivial (take $x_1=y_1$). Suppose the result holds for all $K$-vector spaces of dimension smaller than $n$. Then, by induction, the $K$-vector space generated by $\{y_1,\ldots, y_{n-1}\}$ contains elements $\{x_1,\ldots, x_{n-1}\}$ such that each $v(x_i)$ lies in a different $\Gamma_K$-coset. Without loss of generality we may assume 
\[
v(x_{n}) > v(x_{n-1}) +\Gamma_K >\cdots > v(x_{1}) +\Gamma_K.  
\]
If either $v(y_n)>v(x_{n-1})+\Gamma_K$, $v(x_1)>v(y_n)+\Gamma_K$, or 
\[
v(x_{m+1}) > v(y_n) +\Gamma_K > v(x_{m}) +\Gamma_K,   
\]
for some $1\leqslant m<n-1$, then we are done by setting $x_n\coloneqq y_n$. Otherwise, there are $c\in K$ and $1\leqslant m\leqslant n-1$ such that $v(cy_n) = v(x_{m})$. By assumption, let $a\in K$ be such that $cy_n/x_{m}=x+a$ with $v(x)>\Gamma_K$. Therefore, for $b\coloneqq cy_n-ax_{m}$ satisfies $v(b)>v(x_{m})+\Gamma_K$. If $v(b)$ is in a different $\Gamma_K$-coset than every $v(x_i)$ for $i\in \{1,\ldots, n-1\}$, 
we are done by setting $x_n\coloneqq b$. Otherwise, there are $c'\in K$ and $m'>m$ such that $v(c'b)=v(x_{m'})$. Following the same procedure, there is $a'\in K$ such that 
\[
v(c'b-a'x_{m'})> v(x_{m'})+\Gamma_K.   
\]
 
Again, for $b'\coloneqq c'b-a'x_{m'}$, if $v(b')$ lies in a different $\Gamma_K$-coset than every $x_i$ for $i\in\{1,\ldots, n-1\}$, we are done. Otherwise, there must be $c''\in K$ such that $v(c''b')=v(x_{m''})$ for $m''\in\{1,\ldots, n-1\}$ such that either $m''>m'$. Iterating this argument at most $n-m$ times, one finds a $K$-linear combination $x_n\coloneqq a_{n}y_n+\sum_{i=1}^{n-1} a_{i}x_{i}$ such that $v(x_n)$ is in a different $\Gamma_K$-coset than every $v(x_i)$ for $i\in \{1,\ldots, n-1\}$. 
\end{proof}

Let $(K\subseteq L,v)$ be a valued field extension. Let $G$ be the convex hull of $\Gamma_K$ in $\Gamma_L$ and $w$ be the valuation on $L$ obtained by composing $v$ with the canonical quotient map $\Gamma_L\to \Gamma_L/G$. Let us denote $k_K^w$ and $k_L^w$ the residue fields of $(K,w)$ and $(L,w)$. As $w$ is trivial on $K$, $K\cong k_K^w$. 

An element $a\in L$ is \emph{limit over $K$} if the extension $K(a)|K$ is an immediate extension. We let the reader check that if $K$ is a $p$-adically closed valued field and $a$ is limit over $K$, then the type $tp(a/K)$ is not definable. 

 \begin{theorem}[{\cite[Part (a) of the main Theorem]{delon}}]\label{thm:delon} Suppose $(K\prec L,v)$ is a valued field extension of Henselian valued fields of characteristic 0 and let $w$ be as above. If the canonical embedding $k_K^w\to k_L^w$ is an isomorphism, then $K\prec L$ is stably embedded in $\cL_{\Div}$. 
\end{theorem}

We are ready to show Theorem \ref{thm:valued-charac} for $\PCF_d$:

\begin{proof}[Proof of Theorem \ref{thm:valued-charac} for $\PCF_d$:] 
$(1)\Rightarrow (2)$: That $\Gamma_K\prec \Gamma_L$ is stably embedded follows again by Remark \ref{rem:stablyemb}. It remains to show that the extension is $vs$-defectless. By Lemma \ref{lem:stpart}, it suffices to show that every element $y\in L$ is of the form $x+a$ for $a\in K$ and $x\in L$ such that $|v(x)|>\Gamma_K$. Every element in $y\in K$ is of such form taking $x=0$, so we may suppose $y\in L\setminus K$. If $|v(y)|>\Gamma_K$ take $a=0$. Otherwise, since by Corollary \ref{cor:pres} $\Gamma_L$ is an end extension of $\Gamma_K$, we must have $v(y)\in v(K)$. Suppose there is no $a\in K$ such that $|v(y-a)|>\Gamma_K$. Thus, for every $a\in K$, $v(y-a)\in \Gamma_K$. But this implies that $y$ is a limit over $K$, which contradicts that $K$ is stably embedded in $L$. This shows the extension is $vs$-defectless.  

\

$(2)\Rightarrow (1)$: Since the pair is $vs$-defectless, there are no limit points in $L$ over $K$. Moreover, since $\Gamma_K\prec \Gamma_L$ is stably embedded, by Corollary \ref{cor:pres}, it is an end extension of $\Gamma_K$. The same argument as in the previous implication shows that every element $y\in L$ is of the form $x+a$ for $a\in K$ and $x\in L$ such that $|v(x)|>\Gamma_K$. In particular, the convex hull of $\Gamma_K$ in $\Gamma_L$ is $\Gamma_K$. Let us show that $k_L^w$ is isomorphic to $K$. For all $y\in L\setminus K$ such that $w(y)=0$, there is a unique $a\in K$ such that $v(y-a)>\Gamma_K$. Therefore $\res_w(y)=a$, which shows that $k_L^w$ is in bijection with $K$. The result now follows from \ref{thm:delon}. 
\end{proof}

We summarize the above results in the following theorem. 

\begin{theorem}\label{thm:definabletypes} Let $T$ be one of the following theories 
\begin{enumerate}[$(i)$]
\item a complete o-minimal theory;
\item $\CODF$, Presburger arithmetic, the theory of a finite extension of $\Q_p$ ($\PCF_d$), a completion of $\ACVF$,  $\RCVF$. 
\end{enumerate}
Then, the class $\SE(T)$ is an elementary class in $\cL_P$. 
\end{theorem}

\begin{question} Is there a natural characterization of the class of complete NIP theories $T$ for which $\SE(T)$ is $\cL_P$-elementary?
\end{question} 

Note that there are NIP $\cL$-theories $T$ for which the class $\SE(T)$ is not $\cL_P$-elementary. The following example is due to L. Newelski. 

\begin{example}\label{example:newelski}
Let $(a_i)_{i<\omega}$ be a sequence of irrational numbers $a_i$ such that the limit $\lim_{i\to \infty} a_i=c$ is again irrational an bigger than every $a_i$. Consider the structure $M=(\Q,<, (P_{a_i})_{i<\omega})$ where $P_{a_i}$ is a unary predicate interpreted as the cut $\{x\in \Q \mid x<a_i\}$. The theory $T=\Th(M)$ is NIP as any NIP theory extended by externally definable sets is NIP (see \cite[Section 3.3]{simon}). Now let $N=M\cup \{b_i\mid i<\omega\}$ where $b_i$ realizes the definable type over $M$ determined by the set of formulas 
\[
\{c<x \mid M\models P_{a_i}(c)\} \cup \{x<c\mid M\models \neg P_{a_i}(c) \}. 
\]
Then, $(N,M)$ is a stably embedded pair. Let $(N^*,M^*)$ be an ultrapower of $(N,M)$ over a non-principal ultrafilter $F$ over $\omega$. Let $b$ be the class of the sequence $(b_i)_{i<\omega}$ modulo $F$. We let the reader check that the set $X=\{x\in M^* \mid x<b\}$ is not definable in $M^*$, hence the pair $(N^*,M^*)$ is not stably embedded. This shows $\SE(T)$ is not $\cL_P$-elementary. 
\end{example}

\section{Uniform definability via classes of pairs}\label{sec:UD}

The following theorem provides an abstract criterion for a theory $T$ to have uniform definability of types. 

\begin{theorem}\label{thm:UD2} Suppose there is an $\cL_P$-elementary class $\mathcal{C}$ such that  
\begin{enumerate}[$(i)$]
\item if $(N,M)\in \mathcal{C}$, then $M\models T$;
\item if $(N,M)\in \mathcal{C}$, then $N$ is an $\cL$-substructure of an $\cL$-elementary extension $N'$ of $M$;
\item if $(N,M)\in \mathcal{C}$, and $a$ is a finite tuple in $N$, then $tp(a/M)$ is definable (where $tp(a/M)$ is defined with respect to $N'$); 
\item for every small model $M\models T$ and every definable type $p\in \D{x}(M)$, there is a pair $(N,M)\in \mathcal{C}$ and $a\in \sS_x(N)$ such that $p=tp(a/M)$.
\end{enumerate}
Then $T$ has uniform definability of types. In particular, definable types are pro-definable in $T$ in any reduct of $\cL^{\eq}$ in which $T$ has elimination of imaginaries.   
\end{theorem}

\begin{proof} 
Fix a partitioned $\cL$-formula $\varphi(x;y)$. Let $(\psi_i(y,z_i))_{i\in I}$ be an enumeration of all $\cL$-formulas having $y$ among their free variables. Suppose for a contradiction that no formula $\psi_i(y,z_i)$ provides a uniform definition for $\varphi$. This implies, by Lemma \ref{lem:UD_finite}, that for every finite subset $J\subseteq I$ there are a model $M_J$ of $T$ and a type $q_{J}(x)\in \D{x}(M_J)$ such that no formula $\psi_i$ is a $\varphi$-definition for $q_J$. Consider for every $i\in I$ the $\cL_P$-formula $\theta_i(x)$ 
\[
(\forall z_i\in P)(\exists y\in P)(\neg(\varphi(x,y)\leftrightarrow\psi_i(y,z_i))).
\]
Let $\Sigma(x)\coloneqq\{\theta_i(x)\mid i\in I\}\cup T_\mathcal{C}$, where $T_\mathcal{C}$ is an $\cL_P$-axiomatization of $\mathcal{C}$. Let us show that $\Sigma(x)$ is consistent. Let $\Sigma_0$ be a finite subset of $\Sigma$ and let $J\coloneqq \{i\in I\mid \theta_i(x)\in \Sigma_0\}$. By assumption, there is some $q=q_J(x)\in \D{x}(M_J)$ such that no $\psi_i$ is a $\varphi$-definition for $q$. By condition $(iv)$, let $(N,M_J)$ be an element of $\mathcal{C}$ and $a\in \sS_x(N)$ be such that $q=tp(a/M)$. By the choice of $a$, we have that $\theta_i(a)$ holds for all $i\in J$. This shows that $\Sigma_0$ is consistent. Thus, by compactness, $\Sigma$ is consistent. 

Let $(N', M')$ be an element in $\mathcal{C}$ and $a\in N'$ be a realization of $\Sigma$. By the definition of $\Sigma$, the type $tp(a/M')$ is not definable, which contradicts condition $(iii)$. 
\end{proof}

\begin{corollary}\label{thm:UD} Let $T$ be such that
\begin{enumerate}[$(i)$]
\item $\SE(T)$ is $\cL_P$-elementary; 
\item for every small model $M\models T$ and every definable type $p\in \D{x}(M)$, there is a stably embedded $(N,M)$ such that $P$ is realized in $N$.
\end{enumerate}
Then $T$ has uniform definability of types. In particular, definable types are pro-definable in $T$ in any reduct of $\cL^{\eq}$ in which $T$ has elimination of imaginaries.   
\end{corollary}
\begin{proof} This follows directly by Theorem \ref{thm:UD2} by taking $\mathcal{C}=\SE(T)$. Indeed, note that condition $(i)-(iii)$ of \ref{thm:UD2} are trivially satisfied, and condition $(iv)$ corresponds to the present assumption $(ii)$.  
\end{proof}

\begin{theorem}\label{thm:UDlist} The following theories have uniform definability of types (in one of their natural languages): 
\begin{enumerate}
    \item Any complete stable theory; 
    \item Any complete o-minimal theory; 
    \item Presburger arithmetic; 
    \item The theory of a finite extension of $\Q_p$;
    \item $\RCVF$;
    \item any completion of $\ACVF$;
    \item $\CODF$. 
\end{enumerate}
\end{theorem}

\begin{proof} (1) This is a well-known result (see B. Poizat's paper \cite{poizat}). Let us give two short different proofs. By \cite[Proposition 3.2]{pillay-hrushovski2011}, generically stable types are uniformly definable in any NIP theory. This implies the result since in a stable theory all types are generically stable. Alternatively, use Corollary \ref{thm:UD}: condition $(i)$ is trivially satisfied and for condition $(ii)$ take $N=\cU$.  

Let $T$ be one of the theories form (2)-(6). We apply Corollary \ref{thm:UD}. Condition $(i)$ is granted by Theorem \ref{thm:definabletypes}. For condition (ii), let $M$ be a model of $T$, $p\in \D{x}(M)$ and $a\in \sS_x(\cU)$ be a realization of $p$. If $T$ is o-minimal, by Marker-Stenhorn's theorem (Theorem \ref{thm:MS}), we can take $N\coloneqq M(a)$ since $(M(a),M)$ is already stably embedded. For $T$ a theory from (3)-(6) we take $N\coloneqq \acl(Ma)$. Note that in these cases $N$ is indeed a model of $T$ and the extension $(N,M)$ is stably embedded by Lemma \ref{lem:general_def_types}. This completes the result for $(2)-(6)$. 


For $\CODF$, the structure $N\coloneqq\acl(Ma)$ is a real closed field extension of $M$ but it is not necessarily a model of $\CODF$. To ensure condition $(ii)$ we can apply the following idea of E. Kaplan. Let $N'$ be a sufficiently large real closed field containing $a$ and such that the extension $(N',M)$ is a stably embedded extension of real closed fields. By \cite[Proposition 4.11]{for-kaplan}, there is a derivation $\delta$ on $N'$ extending the derivation on $M$ such that $(N',\delta)$ is a model of $\CODF$. By \cite[Proposition 3.6]{brouette}, the extension of models of $\CODF$ $(N',M)$ is stably embedded (in the language of $\CODF$).

Alternatively, if $\cL$ denotes the language of ordered rings, $\cL_\delta$ its extension by the derivation, and $\cL_{\delta,P}$ the extension of $\cL_\delta$ by a new unary predicate $P$, one can apply Theorem \ref{thm:UD2} to the following $\cL_{\delta,P}$-elementary class of pairs $(N,M)$: 
\begin{enumerate}
    \item $M\models \CODF$;
    \item $(N,M)$ is an $\cL$-stably embedded pair of real closed fields;
    \item $N$ is an $\cL_\delta$-substructure of an $\cL_\delta$-elementary extension of $M$;
\end{enumerate}
We leave the details of this second approach to the reader.  
\end{proof}

\begin{theorem}\label{thm:prodef} Let $T$ be an $\cL$-theory listed in Theorem \ref{thm:UD} and $\cL'$ be any reduct of $\cL^{\eq}$ in which $T$ has elimination of imaginaries. Then, definable types in $T$ are pro-definable in $\cL'$.  
\end{theorem}

\begin{proof} This follows directly by Theorem \ref{thm:UD} and Proposition \ref{prop:pro_def_eq}. 
\end{proof}

We expect similar results hold for theories of (tame) valued fields with generic derivations as defined in \cite{cubi-point} and of o-minimal fields with a generic derivation as defined in \cite{for-kaplan}.

\begin{corollary} Let $T$ be either an o-minimal theory, a completion of $\ACVF$ or $\RCVF$. The bounded completion of a definable set is pro-definable in $\cL^{\eq}$. In the two latter cases, the orthogonal completion of a definable set is pro-definable in $\cL^{\eq}$
\end{corollary}

\begin{proof} 
Both cases are similar, we will just work with the bounded case and leave the others to the reader. By Corollary \ref{cor:pro_def_eq} it suffices to show that $\widetilde{X}(M)$ is $\ast$-definable inside $\D{X}(M)$, for a given $M$-definable set $X$. Suppose $X\subseteq \sS_x(M)$. For every formula $\varphi(x,y,z)$ where $y$ is a $\Gamma$-variable, for every $a\in \sS_z(M)$ such that $\varphi(x,y,a)$ defines a function $f_a\colon X\to \Gamma(M)$, and for every $p(x)\in \widetilde{X}(M)$, there is $\gamma\in \Gamma(M)$ such that $p(x)$ contains the formula $-\gamma< f_a(x) <\gamma$. Let $\psi(x,y',z)$ be the formula 
\[
(\forall y)(\varphi(x,y,z) \rightarrow (-y'<y<y')), 
\]
and $d(\psi)(y',z,z_\psi)$ be its uniform definition. Let $c(p,\psi)$ be the canonical parameter such that $d(\psi)(y',z,c(p,\psi))$ is the $\psi$-definition of $p$. Let $\theta(z_\psi)$ denote the formula $(\forall z)(\exists y') d(\psi)(y',z,z_\psi)$. Then $p\in \widetilde{X}(M)$ if and only if for each formula $\varphi$ and associated formulas $\psi$ and $\theta$ as above   
\[
M\models \theta(c(p,\psi)),  
\]
which shows $\widetilde{X}(M)$ is an $\ast$-definable subset of $\D{X}(M)$. 
\end{proof}

The following question remains open.  

\begin{question*}
Can one characterize NIP theories (or $dp$-minimal theories) having uniform definability of types (resp. definable types are pro-definable)? 
\end{question*}


\bibliographystyle{siam}
\bibliography{biblio}

\end{document}